\documentclass[12pt, reqno]{amsart}
\usepackage{amsmath, amsthm, amscd, amsfonts, amssymb, graphicx, color}
\usepackage[bookmarksnumbered, colorlinks, plainpages]{hyperref}
\hypersetup{colorlinks=true,linkcolor=red, anchorcolor=green, citecolor=cyan, urlcolor=red, filecolor=magenta, pdftoolbar=true}

\usepackage{amsfonts}
\usepackage{color,graphicx,shortvrb}
\usepackage{enumerate}
\usepackage{amsmath}
\usepackage{amssymb}
\usepackage{graphicx}

\usepackage{enumerate}
\usepackage{amsmath}
\usepackage{amssymb}

\usepackage[latin 1]{inputenc}

\usepackage{amsmath}

\newtheorem{theorem}{Theorem}[section]

\newtheorem{corollary}[theorem]{Corollary}
\newtheorem{lemma}[theorem]{Lemma}

\newtheorem{problem}[theorem]{Problem}
\newtheorem{proposition}[theorem]{Proposition}

\newtheorem{remark}[theorem]{Remark}

\usepackage{centernot}

\usepackage{tikz}
\usetikzlibrary{arrows,matrix}

\begin{document}

\title[A generalized \v{S}emrl's theorem for weak-2-local derivations]{On a generalized \v{S}emrl's theorem for weak-2-local derivations on $B(H)$}

\date{\today}

\author[J.C. Cabello]{Juan Carlos Cabello}

\author[A.M. Peralta]{Antonio M. Peralta}

\address{Departamento de An{\'a}lisis Matem{\'a}tico, Universidad de Granada,
Facultad de Ciencias 18071, Granada, Spain}
\email{jcabello@ugr.es, aperalta@ugr.es}

\thanks{Authors partially supported by the Spanish Ministry of Economy and Competitiveness and European Regional Development Fund project no. MTM2014-58984-P and Junta de Andaluc\'{\i}a grants FQM375 and FQM290.}

\keywords{derivation; 2-local linear map; 2-local symmetric maps; 2-local $^*$-derivation; 2-local derivation; weak-2-local derivation}

\subjclass[2010]{47B49, 46L05, 46L40, 46T20, 47L99}


\begin{abstract} We prove that, for every complex Hilbert space $H$, every weak-2-local derivation on $B(H)$ or on $K(H)$ is a linear derivation. We also establish that every weak-2-local derivation on an atomic von Neumann algebra or on a compact C$^*$-algebra is a linear derivation.
\end{abstract}

\maketitle

\section{Introduction}

Let $\mathcal{S}$ be a subset of the space $L(X,Y)$ of all linear
maps between Banach spaces $X$ and $Y$. Following \cite{AyuKuday2012,AyuKuday2014} and \cite{CaPe2015}, we shall say that a (non-necessarily linear nor continuous) mapping $\Delta : X\to Y$ is a \emph{weak-2-local $\mathcal{S}$ map} (respectively, a \emph{2-local $\mathcal{S}$-map}) if for each $x,y\in X$ and $\phi\in Y^{*}$ (respectively, for each $x,y\in X$), there exists $T_{x,y,\phi}\in \mathcal{S}$, depending on $x$, $y$ and $\phi$ (respectively, $T_{x,y}\in \mathcal{S}$, depending on $x$ and $y$), satisfying $$\phi \Delta(x) = \phi T_{x,y,\phi}(x), \hbox{ and  }\phi \Delta(y) = \phi T_{x,y,\phi}(y)$$ (respectively, $ \Delta(x) =  T_{x,y}(x),$ and $  \Delta(y) = T_{x,y}(y)$).\smallskip

When $A$ is a Banach algebra and $\mathcal{S}$ is the set of derivations (respectively, homomorphisms or automorphisms) on $A$, weak-2-local $\mathcal{S}$ maps on $A$ are called \emph{weak-2-local derivations} (respectively, \emph{weak-2-local homomorphisms} or \emph{weak-2-local automorphisms}). \emph{2-local $^*$-derivations} and \emph{2-local $^*$-homomorphisms} on C$^*$-algebras are similarly defined. We recall that a $^*$-derivation on a C$^*$-algebra $A$ is a derivation $D: A\to A$ satisfying $D(a^*) = D(a)^*$ ($a\in A$).\smallskip

The notion of 2-local derivations goes back, formally, to 1997 when P. \v{S}emrl introduces the formal definition and proves that, for every infinite dimensional separable Hilbert space $H$, every 2-local automorphism (respectively, every 2-local derivation) on $B(H)$ is an automorphism (respectively, a derivation). Sh. Ayupov and K. Kudaybergenov proved that \v{S}emrl's theorem also holds for arbitrary Hilbert spaces \cite{AyuKuday2012}. In 2014, Ayupov and Kudaybergenov prove that every 2-local derivation on a von Neumann algebra is a derivation (see \cite{AyuKuday2014}).\smallskip

Results on weak-2-local maps are even more recent. In a very recent contribution, M. Niazi and the second author of this note prove the following generalization of the previously mentioned results.

\begin{theorem}\label{t NiPe Thm 3.10}{\rm\cite[Theorem 3.10]{NiPe2015}} Let $H$ be a separable complex Hilbert space. Then
every {\rm(}non-necessarily linear nor continuous{\rm)} weak-2-local $^*$-derivation on $B(H)$ is linear and a $^*$-derivation.$\hfill\Box$
\end{theorem}

The same authors prove that for finite dimensional C$^*$-algebras the conclusions are stronger:

\begin{theorem}\label{t NiPe Cor 2.13}{\rm\cite[Corollary 2.13]{NiPe2015}} Every weak-2-local derivation on a finite dimensional C$^*$-algebra is a linear derivation.$\hfill\Box$
\end{theorem}

Let $\Delta :A \rightarrow B$ be a mapping between C$^*$-algebras. We consider a new mapping $\Delta^{\sharp}: A \rightarrow B$ given by $\Delta^{\sharp} (x):= \Delta (x^*)^*$ ($x \in A$).  Obviously, $\Delta^{\sharp\sharp}=\Delta$, $\Delta(A_{sa}) \subseteq B_{sa}$ for every $\Delta$ satisfying  $\Delta =\Delta^\sharp$, where $A_{sa}$ and $B_{sa}$ denote the self-adjoint parts of $A$ and $B$, respectively. The mapping $\Delta$ is linear if and only if $\Delta^\sharp$ enjoys the same property. The mapping $\Delta$ is called \emph{symmetric} if $\Delta^{\sharp} = \Delta$ (equivalently, $\Delta (x^*)= \Delta(x)^*$, for all $x\in X$). Henceforth, the set of all symmetric maps from $A$ into $B$ will be denoted by $\mathcal{S}(A,B)$. Weak-2-local $\mathcal{S}(A,B)$ maps between $A$ and $B$ will be called \emph{weak-2-local symmetric maps}, while 2-local $L(A,B)$ maps between $A$ and $B$ will be called \emph{weak-2-local linear maps}.\smallskip

The study on weak-2-local maps has been also pursued in \cite{CaPe2015}, where we obtained that every weak-2-local symmetric map between C$^*$-algebras is linear (see \cite[Theorem 2.5]{CaPe2015}). Among the consequences of this result, we also establish that every weak-2-local $^*$-derivation on a general C$^*$-algebra is a (linear) $^*$-derivation (cf. \cite[Corollary 2.10]{CaPe2015}).\smallskip

One of the main problems that remains unsolved in this line reads as follows:

\begin{problem}\label{problem 1 new}
Is every weak-2-local derivation on a general C$^*$-algebra $A$ a derivation? \end{problem}

We shall justify later that every weak-2-local derivation $\Delta$ on $A$ writes as a linear combination $\Delta =\Delta_1 +\Delta_2$, where $\Delta_1=\frac{\Delta + \Delta^\sharp}{2}$ and $\Delta_2=\frac{\Delta - \Delta^\sharp}{2 i}$ are weak-2-local derivations and symmetric maps. Thus, we shall deduce that the above Problem \ref{problem 1 new} is equivalent to the following question.

\begin{problem}\label{problem 2 new}
Let $\Delta: A\to A$ be a weak-2-local derivation on a C$^*$-algebra which is also a symmetric map {\rm(}i.e. $\Delta^{\sharp} = \Delta${\rm)}. Is $\Delta$ a weak-2-local symmetric map? -- or, equivalently, Is $\Delta$ a linear derivation?
\end{problem}

The above problems are natural questions arisen in an attempt to generalize the above mentioned results by \v{S}emrl's \cite{Semrl97} and Ayupov and Kudaybergenov \cite{AyuKuday2012, AyuKuday2014}. Both remain open even in the intriguing case of $A= B(H)$.\smallskip

In this paper we provide a complete positive answer to both problems in several cases. In Theorem \ref{t final} we prove that every weak-2-local derivation on $A=B(H)$ is a linear derivation. This generalizes the results in  \cite{Semrl97}, \cite{AyuKuday2012, AyuKuday2014} and \cite{NiPe2015}. We also establish that this \emph{weak-2-local stability} of derivations is also true when $A$ coincides with $K(H)$ (see Theorem \ref{t final KH}), when $A$ is an atomic von Neuman algebra (cf. Corollary \ref{t final atomic von Neumann}), and when $A$ is a compact C$^*$-algebra (cf. Corollary \ref{t final compact Cstar}). The techniques and arguments provided in this note are completely new compared with those in previous references. The note is divided in two main sections. In Section 2 we establish a certain boundedness principle showing that for each weak-2-local derivation $\Delta$ on $B(H)$, or on $K(H)$, the mappings $a\mapsto p_{_F} \Delta (p_{_F} a p_{_F}) p_{_F}$ are uniformly bounded when $p_{_F}$ runs in the set of all finite-rank projections on $H$ (compare Theorems \ref{thm boundedness of ztildeF} and \ref{thm boundedness of ztildeF compact}). In Section 3 we derive the main results of the paper from an identity principle, which assures that a weak-2-local derivation $\Delta$ on $B(H)$ with $\Delta^{\sharp} = \Delta$, coincide with a $^*$-derivation $D$ if and only if they coincide on every finite-rank projection in $B(H)$ (see Theorem \ref{t id princ}).

\section{Boundedness of weak-2-local derivations on the lattice of projections in $B(H)$}

We recall some basic properties on weak-2-local maps which have been borrowed from \cite{CaPe2015} and \cite{NiPe2014}.

\begin{lemma}\label{l new basic properties 1}{\rm(\cite[Lemma 2.1]{CaPe2015}, \cite[Lemma 2.1]{NiPe2014})} Let $X$ and $Y$ be Banach spaces and let $\mathcal{S}$ be a subset of the space $L(X,Y)$. Then the following properties hold:\begin{enumerate}[$(a)$] \item Every weak-2-local $\mathcal{S}$ map $\Delta: X\to Y$ is 1-homogeneous, that is, $\Delta (\lambda x) =  \lambda \Delta(x)$, for every $x\in X$, $\lambda\in \mathbb{C}$;
\item  Suppose there exists $C> 0$ such that every linear map $T\in \mathcal{S}$ is continuous with $\|T\|\leq C$. Then every weak-2-local $\mathcal{S}$ map $\Delta: X\to Y$ is $C$-Lipschitzian, that is, $\|\Delta(x)-\Delta (y) \|\leq C \|x-y\|$, for every $x,y\in X$;
\item If $\mathcal{S}$ is a (real) linear subspace of $L(X,Y)$, then every (real) linear combination of weak-2-local $\mathcal{S}$ maps is a weak-2-local $\mathcal{S}$ map;
\item Suppose $A$ and $B$ are C$^*$-algebras and $\mathcal{S}$ is a real linear subspace of $L(A,B)$. If a mapping $\Delta :A \rightarrow B$ is a weak-2-local $\mathcal{S}$ map then for each $\varphi\in B_{sa}^*$ and every $x,y\in A$, there exists $T_{x,y,\varphi}\in \mathcal{S}$ satisfying $\varphi \Delta(x) = \varphi T_{x,y,\varphi}(x)$ and $\varphi \Delta(y) = \varphi T_{x,y,\varphi}(y)$.
\item Suppose $A$ and $B$ are C$^*$-algebras and $\mathcal{S}$ is a real linear subspace of $L(A,B)$ with $\mathcal{S}^{\sharp} = \mathcal{S}$ (in particular when $\mathcal{S} = \mathcal{S}(A,B)$ is the set of all symmetric linear maps from $A$ into $B$). Then a mapping $\Delta :A \rightarrow B$ is a weak-2-local $\mathcal{S}$ map if and only if $\Delta^\sharp$ is a weak-2-local $\mathcal{S}$ map. $\hfill\Box$
\end{enumerate}
\end{lemma}

Henceforth, $H$ will denote an arbitrary complex Hilbert space. The symbols $B(H)$ and $K(H)$ will denote the C$^*$-algebras of all bounded and compact linear operators on $H$, respectively. If $H$ is finite dimensional, then every weak-2-local derivation on $B(H)$ is a linear derivation (compare Theorem \ref{t NiPe Cor 2.13}). We may therefore assume that $H$ is infinite dimensional.\smallskip

Following standard notation, an element $x$ in a C$^*$-algebra $A$ is said to be \emph{finite} (respectively, \emph{compact}) in $A$, if the wedge operator
$x\wedge x:A \to A$, given by $x\wedge x (a)=xax$, is a finite-rank (respectively, compact) operator on $A$. It is known that the ideal $\mathcal{F} (A)$ of finite elements in $A$ coincides with $\hbox{Soc} (A)$, the \emph{socle} of $A$, that is, the sum of all minimal right (equivalently left) ideals of $A$, and that
$\mathcal K (A)=\overline{\hbox{Soc}(A)}$ is the ideal of compact
elements in $A$. Moreover, if $H$
is a Hilbert space, then $\mathcal F (\mathcal L (H))=\mathcal F
(H)$ and $\mathcal K (\mathcal L (H)) =\mathcal K (H)$ are the
ideals of finite-rank and compact elements in $B(H)$, respectively.\smallskip

Suppose $\Delta : B(H) \to B(H)$ is a weak-2-local derivation. By \cite[Lemma 3.4]{NiPe2015} we know that $\Delta (K(H))\subseteq K(H)$ and $\Delta|_{K(H)} : K(H)\to K(H)$ is a weak-2-local derivation. Proposition 3.1 in \cite{NiPe2015} proves that $\Delta (a+ b) = \Delta (a) +\Delta (b)$, for every $a,b\in \mathcal{F} (H)$.

\begin{lemma}\label{lema NIPe finite ranks}{\rm\cite[Lemma 3.4 and Proposition 3.1]{NiPe2015}} Let $\Delta : B(H) \to B(H)$ be a weak-2-local derivation. Then $\Delta|_{\mathcal{F} (H)} : \mathcal{F} (H) \to B(H)$ is linear. $\hfill\Box$
\end{lemma}

Let us revisit some basic facts on commutators. We recall that every derivation on a C$^*$-algebra is continuous (cf. \cite[Lemma 4.1.3]{S}). A celebrated result of S. Sakai establishes that every derivation on a von Neumann algebra $M$ is inner, that is, if $D: M\to M$ is a derivation then there exists $z\in M$ such that $D(x) = [z,x] = z x - x z$ for every $x\in M$ (see \cite[Theorem 4.1.6]{S}). The element $z$ given by Sakai's theorem is not unique; however, we can choose $z$ satisfying $\|z\| \leq \|D\|$. \smallskip

Let us consider two elements $z,w$ in a C$^*$-algebra $A$ such that the derivations $[z,.]$ and $[w,.]$ coincide as linear maps on $A$. Since $[z,x]= [w,x]$ for every $x\in A$, we deduce that $z-w$ lies in the center of $A$. The reciprocal statement is also true, therefore $[z,.]=[w,.]$ on $A$ if and only if $z-w$ lies in the center, $Z(A)$, of $A$. It is known that a derivation of the form $[z,.]$ is symmetric (i.e. a $^*$-derivation) if and only if $z= w + c$, where $w=-w^*$ and $c$ lies in the center of $A$.\smallskip

From now on, the set of all finite dimensional subspaces of $H$ will be denoted by $ \mathfrak{F} (H).$ We consider in $\mathfrak{F} (H)$ the natural order given by inclusion. For each $F\in \mathfrak{F} (H)$, $p_{_F}$ will denote the orthogonal projection of $H$ onto $F$.

\begin{lemma}\label{l zF} Let $\Delta : B(H) \to B(H)$ be a weak-2-local derivation. For each $F\in \mathfrak{F} (H)$ there exists $z_{_F}\in p_{_F} B(H) p_{_F}$ satisfying $$p_{_F} \Delta (p_{_F} a p_{_F}) p_{_F} = [z_{_F}, p_{_F} a p_{_F}], $$ for every $a\in B(H)$. If $\Delta$ is symmetric {\rm(}i.e. $\Delta^{\sharp} = \Delta${\rm)}, then we can choose $z_{_F}\in p_{_F} B(H) p_{_F}$ satisfying $z_{_F}= - z_{_F}^*.$
\end{lemma}

\begin{proof} Let $F$ be a finite dimensional subspace of $H$. By \cite[Proposition 2.7]{NiPe2014} the mapping $p_{_F} \Delta p_{_F}|_{p_{_F} B(H) p_{_F} } : p_{_F} B(H) p_{_F} \to p_{_F} B(H) p_{_F}$, $a\mapsto p_{_F} \Delta (p_{_F}  a p_{_F} ) p_{_F}$ is a weak-2-local derivation. Having in mind that $p_{_F} B(H) p_{_F} $ is a finite dimensional C$^*$-algebra, we deduce from Theorem \ref{t NiPe Cor 2.13} that $p_{_F} \Delta p_{_F}|_{p_{_F} B(H) p_{_F} }$ is a linear derivation. By Sakai's theorem there exists $z_{_F}\in {p_{_F} B(H) p_{_F} }$ satisfying the desired conclusion.\smallskip

If $\Delta$ is symmetric we can easily check that $p_{_F} \Delta p_{_F}|_{p_{_F} B(H) p_{_F} }$ also is symmetric, and hence a $^*$-derivation on ${p_{_F} B(H) p_{_F} }$. In this case, we can obviously replace $z_{_F}$ with $\frac{z_{_F}-z_{_F}^*}{2}$ to get the final statement in the lemma.
\end{proof}

\begin{remark}\label{remark [zF]}{\rm Let $\Delta : B(H) \to B(H)$ be a weak-2-local derivation with $\Delta^{\sharp} = \Delta$, and let  $F$ be a subspace in $\mathfrak{F} (H)$. It is clear that the element $z_{_F}$ given by the above Lemma \ref{l zF} is not unique. We can consider the set $$[z_{_F}]:=\Big\{ z\in p_{_F} B(H) p_{_F} : z^*=-z \hbox{ and } p_{_F}\Delta p_{_F}|_{p_{_F} B(H) p_{_F} } = [z,. ]\Big\}.$$ Given $z_1,z_2\in [z_{_F}]$ it follows that $z_1 -z_2 \in Z(p_{_F} B(H) p_{_F}) = \mathbb{C} p_{_F}$, and since $(z_1-z_2)^* = -(z_1-z_2)$, it follows that there exists $\lambda\in \mathbb{R}$ such that $z_2 = z_1 + i \lambda p_{_F}$.  It is easy to check that there exists a unique $\widetilde{z}_{_F}\in [z_{_F}]$ satisfying $$\|\widetilde{z}_{_F}\| = \min \{ \|z\| : z\in [{z}_{_F}]\}.$$
}\end{remark}

From now on, given an element $a$ in a C$^*$-algebra $A$, the spectrum of $a$ will be denoted by $\sigma(a)$.
Our next remark gathers some information about the norm of an inner $^*$-derivation on $B(H)$.

\begin{remark}\label{remark Stamfli and norms of derivations}{\rm Let $z$ be an element in $B(H)$. J.G. Stampfli proves in \cite[Theorem 4]{Stampfli} that $$\Big\|[z,.] \Big\| = \inf_{\lambda\in \mathbb{C}} \Big\|z - \lambda Id_{H} \Big\|,$$ where $\|[z,.] \|$ denotes the norm of the inner derivation $[z,.]$ in $B(B(H))$.\smallskip

For a compact subset $K\subset \mathbb{C}$, the radius, $\rho(K)$, of $K$ is the radius of the smallest disk containing $K$. In general, two times the radius of a compact set $K$ does not coincide with its diameter. In general, $2 \rho(K)\geq \hbox{diam} (K)$. However, when $K\subset \mathbb{R}$ or $K\subset i \mathbb{R}$, we can easily see that $2 \rho(K)= \hbox{diam} (K)$.\smallskip

When $z$ is a normal operator in $B(H)$ we further know that $\|[z,.] \| = 2 \rho (\sigma(z))$ (compare \cite[Corollary 1]{Stampfli}). In particular, for each $z$ in $B(H)$ with $z=z^*$ or $z=-z^*$, we have \begin{equation}\label{eq norm [z,.] for skew z} \|[z,.] \| = 2 \rho (\sigma(z)) = \hbox{diam} (\sigma(z))\leq 2 \|z\|.
\end{equation} Let us observe that if $0\in\sigma(z)$, then $\|z\| \leq $diam$(\sigma(z))$, for every $z=\pm z^*$.
}\end{remark}

Given a projection $p$ in a unital C$^*$-algebra $A$ we shall denote by $p^{\perp}$ the projection $1-p$.

\begin{lemma}\label{l F1 in F2} Let $\Delta : B(H) \to B(H)$ be a weak-2-local derivation with $\Delta^{\sharp} = \Delta$. Suppose $F_1$ and $F_2$ are finite dimensional subspaces of $H$ with $F_1\subseteq F_2$. We employ the notation given in Remark \ref{remark [zF]}. Then for each $z_1\in \left[z_{_{F_1}}\right]$ and each $z_2\in \left[z_{_{F_2}}\right]$ we have $$[z_1,.] =[p_{_{F_1}} z_2 p_{_{F_1}},.]$$ as operators on $p_{_{F_1}} B(H) p_{_{F_1}}$. Consequently, there exists a real $\lambda$ {\rm(}depending on $z_1$ and $z_2${\rm)} such that $z_1 +i \lambda  p_{_{F_1}} = p_{_{F_1}} z_2 p_{_{F_1}}$. In particular, diam$(\sigma(z_1))\leq $diam$(\sigma(z_2))$ and diam$(\sigma(z_1))= $diam$(\sigma({z'}_1))$ for every $z_1,z'_1\in [z_{_{F_1}}]$.
\end{lemma}

\begin{proof} By Lemma \ref{l zF} and Remark \ref{remark [zF]} we have $$p_{_{F_1}}\Delta  ({p_{_{F_1}} a p_{_{F_1}} }) p_{_{F_1}}= [z_1, p_{_{F_1}} a p_{_{F_1}} ],$$ and $$p_{_{F_2}}\Delta ({p_{_{F_2}} a p_{_{F_2}} } )p_{_{F_2}}= [z_2, p_{_{F_2}} a p_{_{F_2}} ],$$ for every $a\in B(H)$. Since $p_{_{F_1}} \leq p_{_{F_2}}$, it follows that $$[p_{_{F_1}} z_2 p_{_{F_1}}, p_{_{F_1}} a p_{_{F_1}} ]  = p_{_{F_1}}  [z_2, p_{_{F_1}} a p_{_{F_1}} ] p_{_{F_1}} = p_{_{F_1}} p_{_{F_2}}\Delta ({p_{_{F_1}} a p_{_{F_1}} } )p_{_{F_2}} p_{_{F_1}}$$ $$=p_{_{F_1}}\Delta  ({p_{_{F_1}} a p_{_{F_1}} }) p_{_{F_1}}= [z_1, p_{_{F_1}} a p_{_{F_1}} ]$$ for every $a\in B(H)$, which proves the first statement in the lemma.\smallskip

Since $Z(p_{_{F_1}} B(H) p_{_{F_1}}) =\mathbb{C} p_{_{F_1}}$, $z_1^*=-z_1$ and $z_2^*=-z_2$, there exists $\lambda\in \mathbb{R}$ such that $z_1 +i \lambda  p_{_{F_1}} = p_{_{F_1}} z_2 p_{_{F_1}}$ (compare Remark \ref{remark [zF]}).\smallskip

By Remark \ref{remark Stamfli and norms of derivations} we have $$\hbox{diam}(\sigma(z_1)) = \| [z_1, .]\| =\| [p_{_{F_1}} z_2 p_{_{F_1}},.]|_{(p_{_{F_1}}B(H) p_{_{F_1}})} \|  $$ $$= \|p_{_{F_1}} [ z_2 , p_{_{F_1}} . p_{_{F_1}}] p_{_{F_1}} \| \leq \| [ z_2 , .]\| = \hbox{diam}(\sigma(z_2)) $$
\end{proof}

\begin{proposition}\label{p unboundedness of diamzF passes to subnets} Let $\Delta : B(H) \to B(H)$ be a weak-2-local derivation with $\Delta^{\sharp} = \Delta$. Suppose that the set $\mathcal{D}\hbox{iam} (\Delta)=\left\{ \hbox{diam} (\sigma(w_{_F})) : w_F\in[z_{_F}], \ F\in \mathfrak{F} (H) \right\}$ is unbounded. Then for each $G\in \mathfrak{F} (H)$, the set $$\mathcal{D}\hbox{iam}_{G}^+=\left\{ \hbox{diam} (\sigma(w_{_F})) : w_F\in[z_{_F}], \ F\in \mathfrak{F} (H), \ F\supseteq G \right\}$$ is unbounded.
\end{proposition}

\begin{proof}
Let us fix an arbitrary $G\in \mathfrak{F} (H)$. For each $F\in \mathfrak{F} (H)$ we can find $K\in \mathfrak{F} (H)$ with $G,F\subseteq K$. Applying Lemma \ref{l F1 in F2} we have $$\hbox{diam} (\sigma(w_{_F})), \hbox{diam} (\sigma(w_{_G})) \leq \hbox{diam} (\sigma(w_{_K})),$$ for every $w_F\in[z_{_F}]$, $w_G\in[z_{_G}]$ and $w_K\in[z_{_K}]$. The unboundedness of $\mathcal{D}\hbox{iam}$ implies the same property for $\mathcal{D}\hbox{iam}_{G}^+$.
\end{proof}

\subsection{An identity principle for weak-2-local derivations}

Let $\Delta : B(H) \to B(H)$ be a weak-2-local derivation with $\Delta^{\sharp} = \Delta$. Suppose there exists $G\in \mathfrak{F} (H)$ such that the set $$\mathcal{D}\hbox{iam}_{G^{\perp}}^-=\left\{ \hbox{diam} (\sigma(w_{_F})) : w_F\in[z_{_F}], \ F\in \mathfrak{F} (H), \ p_{_{F}} \leq p_{_{G}}^{\perp} \right\}$$ is bounded. For each $F\in \mathfrak{F} (H)$ with $p_{_{F}} \leq p_{_{G}}^{\perp}$, the element $\widetilde{z}_{_F}$ has been chosen to satisfy $$\| \widetilde{z}_{_F}\| \leq \hbox{diam}(\sigma(\widetilde{z}_{_F}))\leq  2 \| \widetilde{z}_{_F}\|.$$ Therefore, the net $(\widetilde{z}_{_F})_{_{F\in \mathfrak{F} (H), p_{_{F}} \leq p_{_{G}}^{\perp}}}$ is bounded in $ p_{_{G}}^{\perp} B(H) p_{_{G}}^{\perp}$. By Alaoglu's theorem, we can find $z_0\in p_{_{G}}^{\perp} B(H) p_{_{G}}^{\perp}$ with $z_0=-z_0^*$ and a subnet $(\widetilde{z}_{_F})_{F\in \Lambda}$ converging to $z_0$ in the weak$^*$-topology of $p_{_{G}}^{\perp} B(H) p_{_{G}}^{\perp}$.\smallskip

If the set $\mathcal{D}\hbox{iam} (\Delta)=\mathcal{D}\hbox{iam}_{\{0\}^{\perp}}^-=\left\{ \hbox{diam} (\sigma(w_{_F})) : w_F\in[z_{_F}], \ F\in \mathfrak{F} (H) \right\}$ is bounded, we can similarly define, via Alaoglu's theorem, an element $z_0 = -z_0^*\in B(H)$ which is the weak$^*$-limit of a convenient subnet of $(\widetilde{z}_{_F})_{_{F\in \mathfrak{F} (H)}}.$

\begin{proposition}\label{prop Id Princ 1} Let $\Delta : B(H) \to B(H)$ be a weak-2-local derivation with $\Delta^{\sharp} = \Delta$. Suppose there exists $G\in \mathfrak{F} (H)$ such that the set $$\mathcal{D}\hbox{iam}_{G^{\perp}}^-=\left\{ \hbox{diam} (\sigma(w_{_F})) : w_F\in[z_{_F}], \ F\in \mathfrak{F} (H), \ p_{_{F}} \leq p_{_{G}}^{\perp} \right\}$$ is bounded, and let $z_0\in p_{_{G}}^{\perp} B(H)p_{_{G}}^{\perp} $ {\rm(}$z_0=-z_0^*${\rm)} be the element determined in the previous paragraph. Then $p_{_{G}}^{\perp} \Delta (p) p_{_{G}}^{\perp} = [z_0,p]$, for every projection $p\in \mathcal{F} (H)$ with $p\leq p_{_{G}}^{\perp}$.\smallskip

If the set $\mathcal{D}\hbox{iam} (\Delta)=\mathcal{D}\hbox{iam}_{\{0\}^{\perp}}^-=\left\{ \hbox{diam} (\sigma(w_{_F})) : w_F\in[z_{_F}], \ F\in \mathfrak{F} (H) \right\}$ is bounded, then $\Delta (p) = [z_0,p]$, for every projection $p\in \mathcal{F} (H)$.
\end{proposition}

\begin{proof} Let us fix a finite-rank projection $p\in \mathcal{F} (H)$ with $p\leq p_{_{G}}^{\perp}$. Since $(\widetilde{z}_{_F})_{F\in \Lambda}$ converges to $z_0$ in the weak$^*$ topology of $p_{_{G}}^{\perp}B(H)p_{_{G}}^{\perp}$, where $(\widetilde{z}_{_F})_{F\in \Lambda}$ is the subnet fixed before Proposition \ref{prop Id Princ 1}, there exists $F_0\in \Lambda$ such that $p\leq p_{_{F_0}}$ (we observe that, under these hypothesis, there exits a monotone final function $h:\mathfrak{F} (H)\to \Lambda$ which defines the subnet). The subnet $(\widetilde{z}_{_F})_{F_0\subseteq F\in \Lambda}$ converges to $z_0$ in the weak$^*$ topology of $B(H)$.\smallskip

Clearly, the net $(p_{_F})_{F\in \mathfrak{F} (H)}$ converges to the projection $p_{_{G}}^{\perp}$ in the strong$^*$ topology of $B(H)$. Therefore the subnet $(p_{_F})_{F_0\subseteq F\in \Lambda} \to p_{_{G}}^{\perp}$ in the strong$^*$ topology of $B(H)$. Since for each $F\in \Lambda$ with $F_0\subseteq F$ we have $p\leq p_{_{F_0}}\leq p_{_{F}}$, we deduce, via \cite[Lemma 3.2]{NiPe2015}, Lemma \ref{l zF} and Remark \ref{remark [zF]}, that \begin{equation}\label{eq 1 prop ID princ 1} p_{_F} \Delta (p) p_{_F} = p_{_F} \Delta ( p_{_F} p p_{_F}) p_{_F} = p_{_F} [ \widetilde{z}_{_F} ,p_{_F} p p_{_F}] p_{_F} = [ \widetilde{z}_{_F} , p  ].
\end{equation} It is known that the product of every von Neumann algebra is jointly strong$^*$-continuous on bounded sets (see \cite[Proposition 1.8.12]{S}), we thus deduce that the net $(  p_{_F} \Delta (p) p_{_F} )_{F_0\subseteq F\in \Lambda} \to p_{_{G}}^{\perp} \Delta(p) p_{_{G}}^{\perp}$ in the strong$^*$ topology of $B(H)$, and hence $( p_{_F} \Delta (p) p_{_F} )_{F_0\subseteq F\in \Lambda} \to p_{_{G}}^{\perp} \Delta(p) p_{_{G}}^{\perp}$ also in the weak$^*$ topology (compare \cite[Theorem 1.8.9]{S}). This shows that the left-hand side in \eqref{eq 1 prop ID princ 1} converges to $p_{_{G}}^{\perp} \Delta (p) p_{_{G}}^{\perp} $ in the weak$^*$-topology of $B(H)$.\smallskip

Finally, the separate weak$^*$-continuity of the product of $B(H)$ (cf. \cite[Theorem 1.7.8]{S}) shows that the right-hand side in \eqref{eq 1 prop ID princ 1}  converges to $p_{_{G}}^{\perp} [z_0,p] p_{_{G}}^{\perp}= [z_0, p]$ in the weak$^*$-topology. Therefore, $p_{_{G}}^{\perp} \Delta (p) p_{_{G}}^{\perp} = [z_0,p]$ as we desired. The second statement follows from the same arguments.
\end{proof}

We can state now an identity principle for weak-2-local derivations on $B(H)$.

\begin{theorem}\label{t id princ} Let $\Delta : B(H) \to B(H)$ be a weak-2-local derivation with $\Delta^{\sharp} = \Delta$. Let $p_0\in \mathcal{F} (H)$ be a finite rank projection. Suppose $z_0$ is a skew symmetric element in $(1-p_0)B(H)(1-p_0)$ such that  $(1-p_0) \Delta (p) (1-p_0) = [z_0,p]$, for every finite-rank projection $p\in (1-p_0) B (H)(1-p_0)$. Then $$(1-p_0) \Delta ((1-p_0)a(1-p_0)) (1-p_0)= [z_0,(1-p_0) a(1-p_0) ],$$ for every $a\in B(H)$. If in addition $p_0=0$, then $\Delta = [z_0,.]$ is a linear derivation on $B(H)$.
\end{theorem}

\begin{proof} Let $D: B(H)\to B(H)$ denote the $^*$-derivation defined by $D(a)=[z_0,a]$ ($a\in B(H)$). Lemma \ref{lema NIPe finite ranks} (see also \cite[Lemma 3.4 and Proposition 3.1]{NiPe2015}) assures that $\Delta|_{\mathcal{F} (H)} : \mathcal{F} (H)\to \mathcal{F} (H)$ is a linear mapping. Since every element in $\mathcal{F} (H)$ can be written as a finite linear combination of finite-rank projections in $B(H)$, it follows from our hypothesis that \begin{equation}\label{eq D and Delta conincide on F(H)} (1-p_0) \Delta (1-p_0)|_{ (1-p_0)\mathcal{F} (H)(1-p_0)} = D|_{(1-p_0)\mathcal{F} (H)(1-p_0)}= [z_0,.]|_{(1-p_0)\mathcal{F} (H)(1-p_0)}.
\end{equation}\smallskip

Fix $a$ in $(1-p_0)B(H)(1-p_0)$ and a finite-rank projection $p_1\leq  1-p_0$. Having in mind that $p_1 a p_1 + p_1 a p_1^\perp +p_1^\perp a p_1\in (1-p_0)\mathcal{F} (H)(1-p_0)$, Lemma 3.2 in \cite{NiPe2015}, and \eqref{eq D and Delta conincide on F(H)}, we conclude that \begin{equation}\label{eq 2 in final thm} p_1 \Delta (a) p_1 = p_1 \Delta (p_1 a p_1 + p_1 a p_1^\perp +p_1^\perp a p_1) p_1 = p_1 [z_0, (p_1 a p_1 + p_1 a p_1^\perp +p_1^\perp a p_1)] p_1.
\end{equation}

The net $(p_{_F})_{\stackrel{ F\in \mathfrak{F} (H)}{\scriptscriptstyle p_{_F}\leq 1-p_0}}$ converges to $1-p_0$ in the strong$^*$ topology of $B(H)$. We deduce from \eqref{eq 2 in final thm} that $$p_{_F} \Delta (a) p_{_F} = p_{_F} [z_0, p_{_F} a + p_{_F}^\perp a p_{_F}] p_{_F},$$ for every $F\in \mathfrak{F}(H)$ with $p_{_F}\leq 1-p_0$. Taking strong$^*$-limits in the above identity, it follows from the joint strong$^*$-continuity of the product in $B(H)$ that $$(1-p_0) \Delta (a) (1-p_0)= [z_0,a],$$ which finishes the proof.
\end{proof}

Our next result is a consequence of Proposition \ref{prop Id Princ 1} and Theorem \ref{t id princ}.

\begin{corollary}\label{c id principle 3} Let $\Delta : B(H) \to B(H)$ be a weak-2-local derivation with $\Delta^{\sharp} = \Delta$. Suppose that one of the following statements holds:
\begin{enumerate}[$(a)$] \item The set $\mathcal{D}\hbox{iam} (\Delta)=\left\{ \hbox{diam} (\sigma(w_{_F})) : w_F\in[z_{_F}], \ F\in \mathfrak{F} (H) \right\}$ is bounded;
\item The set $\left\{ \left\| p_{_F} \Delta p_{_F}|_{p_{_F}B(H) p_{_F}} \right\| :  F\in \mathfrak{F} (H) \right\}$ is bounded.
\end{enumerate} Then $\Delta$ is a linear derivation.
\end{corollary}

\begin{proof} If $\Delta$ satisfies $(a)$, the the conclusion follows straightforwardly from Proposition \ref{prop Id Princ 1} and Theorem \ref{t id princ}. If we assume $(b)$ we simply observe that for each $F\in \mathfrak{F} (H)$ we have $$\| \widetilde{z}_{_F}\| \leq \hbox{diam}(\sigma(\widetilde{z}_{_F}))= \left\|[\widetilde{z}_{_F},.]|_{p_{_F}B(H) p_{_F}}\right\|= \left\| p_{_F} \Delta p_{_F}|_{p_{_F}B(H) p_{_F}} \right\|\leq  2 \| \widetilde{z}_{_F}\|$$ (compare Remarks \ref{remark [zF]} and \ref{remark Stamfli and norms of derivations}).
\end{proof}

The following lemma states a simple property of derivations on $M_n$. The proof is left to the reader.

\begin{lemma}\label{l new tech one} Let $D: M_n \to M_n$ be a $^*$-derivation. Suppose $p_1$ is a rank one projections in $M_n$. If $D(a) = 0$ for every $a= p_1^{\perp} a p_1^{\perp}$ in $M_n$, then there exists $\alpha\in i\mathbb{R}$ such that $D(x) = [\alpha p_1, x]$ for all $x\in M_n$.$\hfill\Box$
\end{lemma}

We state now an infinite dimensional analog of the previous lemma.

\begin{proposition}\label{p new tech one} Let $\Delta: B(H) \to B(H)$ be a weak-2-local derivation with $\Delta^{\sharp} = \Delta$. Suppose $p_0$ is rank one projection in $M_n$ such that $\Delta (a) = 0$ for every $a= p_0^{\perp}a p_0^{\perp}$ in $B(H)$, then there exists $\alpha\in i\mathbb{R}$ such that $\Delta (x) = [\alpha p_0, x]$ for all $x\in B(H)$.
\end{proposition}

\begin{proof} Take a finite rank projection $p\leq p_0^{\perp}$. Since $$\Delta_p = (p+p_0) \Delta (p_0+p)|_{(p_0+p)B(H) (p_0+p)}: (p_0+p)B(H)(p_0+p)\to (p_0+p)B(H) (p_0+p)$$ is a weak-2-local derivation with $\Delta_p^{\sharp} = \Delta_p$ (compare \cite[Proposition 2.7]{NiPe2014}) and $(p_0+p)B(H)(p_0+p) \cong M_m$ for a suitable $m$, we deduce from \cite[Theorem 2.12]{NiPe2015} that $\Delta_p$ is a $^*$-derivation. We also know that $\Delta_p (a) = 0$ for every $a \in (p_0+p)B(H)(p_0+p)$ with $a = p a p$. Lemma \ref{l new tech one} implies the existence of $\alpha(p)\in i\mathbb{R}$, depending on $p$, such that $\Delta_p (x) = [\alpha(p) p_0, x]$ for all $x\in (p_0+p) B(H)(p_0+p)$.\smallskip

We claim that $\alpha (p)$ doesn't depend on $p$. Indeed, let $p_1,p_2$ be finite rank projections with $p_j\leq p_0^{\perp}$. We can find a third finite rank projection $p_3\leq p_0^{\perp}$ such that $p_1,p_2\leq p_3$. We know that $\Delta_{p_j} (x) = [\alpha(p_{j}) p_0, x]$ for all $x\in (p_0+p_j) B(H)(p_0+p_j)$ for all $j=1,2,3$. Since for each $j=1,2$, $$(p_0+p_j) \Delta_{p_3} (p_0+p_j) |_{(p_0+p_j)B(H) (p_0+p_j)} = \Delta_{p_j},$$ we can easily see that $\alpha(p_j) = \alpha (p_3)$ for every $j=1,2$, which proves the claim. Therefore, there exists $\alpha \in i\mathbb{R}$ such that \begin{equation}\label{eq identity alpha p0 finites} (p+p_0) \Delta (x) (p_0+p) = [\alpha p_0, x]
 \end{equation}for all $x\in (p_0+p) B(H)(p_0+p)$ and every finite rank projection $p\leq p_0^{\perp}$. \smallskip

Let us fix $F\in \mathfrak{F} (H)$. We can find another finite rank projection $p_1\leq p_0^{\perp}$ such that $p_{_F}\leq p_0 + p_1$. We have shown that $\Delta_{p_1} = (p_0+p_1) \Delta (p_0+p_1)|_{(p_0+p_1)B(H) (p_0+p_1)} = [\alpha p_0, .]|_{(p_0+p_1)B(H) (p_0+p_1)},$ and hence $\|\Delta_{p_1}\|\leq 2 |\alpha|$. Since $ p_{_F} \Delta_{p_1} p_{_F} |_{p_{_F} B(H)p_{_F} } = p_{_F} \Delta p_{_F} |_{p_{_F} B(H)p_{_F} }$, we can also conclude that $$\left\| p_{_F} \Delta p_{_F}|_{p_{_F}B(H) p_{_F}} \right\|\leq 2 |\alpha|,$$ for every $F\in \mathfrak{F} (H)$. Corollary \ref{c id principle 3} implies that $\Delta$ is a linear $^*$-derivation. The continuity and linearity of $\Delta$ combined with \eqref{eq identity alpha p0 finites} give the desired statement.
\end{proof}

\begin{theorem}\label{t boundedness of diamzF in the orthogonal complement of a finite rank projection} Let $\Delta : B(H) \to B(H)$ be a weak-2-local derivation with $\Delta^{\sharp} = \Delta$. Suppose there exists $G\in \mathfrak{F} (H)$ such that the set $$\mathcal{D}\hbox{iam}_{G^{\perp}}^-=\left\{ \hbox{diam} (\sigma(w_{_F})) : w_F\in[z_{_F}], \ F\in \mathfrak{F} (H), \ p_{_{F}} \leq p_{_{G}}^{\perp} \right\}$$ is bounded. Then $\Delta$ is a linear $^*$-derivation.
\end{theorem}

\begin{proof} Combining Proposition \ref{prop Id Princ 1} and Theorem \ref{t id princ} we deduce the existence of $z_0= -z_0^*$ in $(1-p_{_G}) B(H) (1-p_{_G})$ such that $$(1-p_{_G}) \Delta (a) (1-p_{_G})= [z_0,(1-p_{_G}) a(1-p_{_G}) ],$$ for every $a\in B(H)$. The mapping $\Delta_1 = \Delta -[z_0, .]$ is a weak-2-local derivation on $B(H)$ with $\Delta_1 = \Delta_1^{\sharp},$ and satisfies \begin{equation}\label{eq Delta1 annihilates orthogonal} (1-p_{_G}) \Delta_1 (a) (1-p_{_G}) =0,
 \end{equation}for all $a\in (1-p_{_G}) B(H) (1-p_{_G})$.\smallskip

Let $q_1,\ldots, q_m$ be mutually orthogonal rank one projections such that $p_{_G} = q_1+\ldots+ q_m.$ \smallskip

Let $\{\xi_{j}: j\in J\}$ be an orthonormal basis of $p_{_G}^{\perp} (H)$. For each $j\in J$ we denote by $p_j$ the rank-projection corresponding to the orthogonal projection of $H$ onto $\mathbb{C} \xi_j$. By Proposition 2.7 in \cite{NiPe2014} the mapping $(q_{m} + p_j) \Delta_1 (q_{m} + p_j)|_{_{(q_{m} + p_j) B(H) (q_{m} + p_j)}}$  is a linear $^*$-derivation on $(q_{m} + p_j) B(H) (q_{m} + p_j)$. Therefore, there exists $\displaystyle z_j = \left(
                                 \begin{array}{cc}
                                   \alpha^{j}_{00} & \alpha^{j}_{0j} \\
                                   - \overline{\alpha^{j}_{0j}} & \alpha^{j}_{jj} \\
                                 \end{array}
                               \right)
= - z_j^*\in M_2 (\mathbb{C})$ such that $(q_{m} + p_j) \Delta_1 (q_{m} + p_j) (a) = [z_j , a],$ for every $a\in (q_{m} + p_j) B(H) (q_{m} + p_j).$ We deduce from \eqref{eq Delta1 annihilates orthogonal} that $\alpha^{j}_{jj} =0$ (for every $j$). We have thus defined a family $(\alpha^{j}_{0j})\subset\mathbb{C}$.\smallskip

The same arguments give above show, via \cite[Proposition 2.7]{NiPe2014} and \eqref{eq Delta1 annihilates orthogonal}, that for each finite subset $J_0\subset J$, with $k_0=\sharp J_0$,  and $p_{_{J_0}}= \sum_{j\in J_0} p_{j}$ that \begin{equation}\label{eq finite dime cuts} (q_{m} + p_{_{J_0}}) \Delta_1 (a) (q_{m} + p_{_{J_0}}) = [z_{_{J_0}} , a],
 \end{equation}for all $a\in (q_{m} + p_{_{J_0}}) B(H) (q_{m} + p_{_{J_0}}),$ where $z_{_{J_0}}$ identifies with the $(k_0+1)\times (k_0+1)$ skew symmetric matrix given by $z_{_{J_0}} =\alpha_{00} q_{k_{0}} + \sum_{j\in J_0} \alpha^{j}_{0j} e_{0j} - \overline{\alpha^{j}_{0j}} e_{0j}^*,$ where $e_{0j}$ is the unique minimal partial isometry satisfying $e_{0j} e_{0j}^* = q_{m}$ and $e_{0j}^* e_{0j} = p_j,$ and $\alpha_{00}$ is a suitable complex number.\smallskip

We claim that the family $\sum_{j\in J} |\alpha^{j}_{0j}|^2$ is summable. Indeed, for each finite subset $J_0\subset J$, we can show from \eqref{eq finite dime cuts} and \cite[Lemma 3.2]{NiPe2015} that $$ \sum_{j\in J_0} \alpha^{j}_{0j} e_{0j} + \overline{\alpha^{j}_{0j}} e_{0j}^* = (q_{m} + p_{_{J_0}}) \Delta_1(p_{_{J_0}}) (q_{m} + p_{_{J_0}})= (q_{m} + p_{_{J_0}}) \Delta_1(p_{_G}^{\perp}) (q_{m} + p_{_{J_0}}),$$ and hence $$ \sum_{j\in J_0} |\alpha^{j}_{0j}|^2  = \| (q_{m} + p_{_{J_0}}) \Delta_1(p_{_{J_0}}) (q_{m} + p_{_{J_0}}) \|^2 \leq \|  \Delta_1(p_{_G}^{\perp}) \|^2,$$ which assures the boundedness of the set $\{ \sum_{j\in J} |\alpha^{j}_{0j}|^2 : J_0\subset J \hbox{ finite }\}$ and proves the claim.\smallskip

Thanks to the claim, the element $z_1 = \sum_{j\in J} \alpha^{j}_{0j} e_{0j} - \overline{\alpha^{j}_{0j}} e_{0j}^*$ is a well-defined skew symmetric element in $B(H)$. We further know, from \eqref{eq finite dime cuts}, that \begin{equation}\label{eq finite dime cuts 2} (q_{m} + p_{_{J_0}}) \Delta_1  (a) (q_{m} + p_{_{J_0}}) = (q_{m} + p_{_{J_0}}) [z_{1} , a](q_{m} + p_{_{J_0}}),
\end{equation} for every finite subset $J_0\subset J$, $p_{_{J_0}}= \sum_{j\in J_0} p_{j}$, and every element $a$ in $ p_{_{J_0}} B(H)  p_{_{J_0}}.$ In the case $a= p_{_{J_0}})$ we get $$ (q_{m} + p_{_{J_0}}) \Delta_1  (p_{_{J_0}}) (q_{m} + p_{_{J_0}}) = (q_{m} + p_{_{J_0}}) [z_{1} , p_{_{J_0}} ](q_{m} + p_{_{J_0}}).$$ Lemma 3.2 in \cite{NiPe2015} implies that $$(q_{m} + p_{_{J_0}}) \Delta_1  (p_{_{G}}^{\perp}) (q_{m} + p_{_{J_0}}) = (q_{m} + p_{_{J_0}}) \Delta_1  (p_{_{J_0}} + (p_{_G}^\perp-p_{_{J_0}})) (q_{m} + p_{_{J_0}})$$ $$= (q_{m} + p_{_{J_0}}) \Delta_1  (p_{_{J_0}} ) (q_{m} + p_{_{J_0}}) = (q_{m} + p_{_{J_0}}) [z_{1} , p_{_{J_0}} ](q_{m} + p_{_{J_0}}) $$ $$= (q_{m} + p_{_{J_0}}) [z_{1} , p_{_G}^{\perp} ](q_{m} + p_{_{J_0}}).$$ Letting $p_{_{J_0}} \nearrow p_{_G}^{\perp}$ in the strong$^*$-topology, we get $$ (q_{m} + p_{_G}^{\perp}) \Delta_1  (p_{_{G}}^{\perp}) (q_{m} + p_{_G}^{\perp}) =  (q_{m} + p_{_G}^{\perp}) [z_{1} , p_{_G}^{\perp} ] (q_{m} + p_{_G}^{\perp}) = \widehat{z}_1 = \sum_{j\in J} \alpha^{j}_{0j} e_{0j} + \overline{\alpha^{j}_{0j}} e_{0j}^*.$$ Clearly, $z_1 = q_{m} \widehat{z}_1 p_{_G}^{\perp} - p_{_G}^{\perp} \widehat{z}_1 q_{m}$. Let $p\leq p_{_G}^{\perp}$ be a finite rank projection. We deduce from the last identity that $$q_{m} [{z}_1,p] p =q_{m} \widehat{z}_1 p =  q_{m} \Delta_1 (p_{_G}^{\perp}) p = q_{m} \Delta_1 (p + (p_{_G}^{\perp}-p)) p =  q_{m} \Delta_1 (p) p,$$ where the last equality follows from \cite[Lemma 3.2]{NiPe2015}. We similarly prove $p [{z}_1,p] q_{m}= - p \widehat{z}_1 q_{m} =  p \Delta_1 (p) q_{m}$, and hence, by \eqref{eq Delta1 annihilates orthogonal}, $$ ( q_{m} +p_{_G}^{\perp}) \Delta_1 (p) ( q_{m} +p_{_G}^{\perp}) = ( q_{m} +p_{_G}^{\perp}) [z_1, p] ( q_{m} +p_{_G}^{\perp}).$$

Now, Proposition 3.1 in \cite{NiPe2015} shows that $\Delta_1$ is linear on $\mathcal{F} (H)$. We thus deduce from the above that \begin{equation}\label{eq q0pGperp nanihilates on pGperp}  ( q_{m} +p_{_G}^{\perp}) \Delta_1 (a) ( q_{m} +p_{_G}^{\perp}) = ( q_{m} +p_{_G}^{\perp}) [z_1, a] ( q_{m} +p_{_G}^{\perp}),
\end{equation} for every $a\in p_{_G}^{\perp} \mathcal{F}(H)p_{_G}^{\perp}.$\smallskip

We claim now that $$(q_{m} + p_{_G}^{\perp}) \Delta_1  (a) (q_{m} + p_{_G}^{\perp}) = (q_{m} + p_{_G}^{\perp}) [z_{1} , a](q_{m} + p_{_G}^{\perp}),$$ for every element $a$ in $ p_{_G}^{\perp} B(H) p_{_G}^{\perp}.$ For this purpose, let us fix $a\in p_{_G}^{\perp} B(H) p_{_G}^{\perp},$ and a projection $p_{_{J_0}},$ with $J_0$ a finite subset of $J$. Having in mind that $(q_{m} + p_{_{J_0}}) a + (q_{m} + p_{_{J_0}})^{\perp} a (q_{m} + p_{_{J_0}})\in p_{_G}^{\perp} \mathcal{F}(H)p_{_G}^{\perp}$, a new application of \cite[Lemma 3.2]{NiPe2015} proves that
$$( q_{m} +p_{_G}^{\perp}) [z_1, a] ( q_{m} +p_{_G}^{\perp})\!= \!( q_{m} +p_{_G}^{\perp}) [z_1, (q_{m} + p_{_{J_0}}) a + (q_{m} + p_{_{J_0}})^{\perp} a (q_{m} + p_{_{J_0}})] ( q_{m} +p_{_G}^{\perp})$$ $$=(q_{m} + p_{_{J_0}}) \Delta_1  ((q_{m} + p_{_{J_0}}) a + (q_{m} + p_{_{J_0}})^{\perp} a (q_{m} + p_{_{J_0}})) (q_{m} + p_{_{J_0}}) $$
$$=(q_{m} + p_{_{J_0}}) \Delta_1  (a) (q_{m} + p_{_{J_0}}) .$$ If in the previous identity we let  $p_{_{J_0}} \nearrow p_{_G}^{\perp}$ in the strong$^*$-topology we obtain the equality stated in the claim.\smallskip

The mapping $(q_{m} + p_{_G}^{\perp}) \Delta_1   ( q_{m} + p_{_G}^{\perp})|_{(q_{m} + p_{_G}^{\perp})B(H)(q_{m} + p_{_G}^{\perp})}$ is a weak-2-local derivation on $(q_{m} + p_{_G}^{\perp})B(H)(q_{m} + p_{_G}^{\perp})$ (see \cite[Proposition 2.7]{NiPe2014}). We know from \eqref{eq q0pGperp nanihilates on pGperp} that $(q_{m} + p_{_G}^{\perp}) \Delta_1 (a)  ( q_{m} + p_{_G}^{\perp}) = (q_{m} + p_{_G}^{\perp}) [z_{1} , a](q_{m} + p_{_G}^{\perp}),$ for every element $a$ in $ p_{_G}^{\perp} B(H) p_{_G}^{\perp}.$ We set $$\Delta_2 = (q_{m} + p_{_G}^{\perp}) \Delta_1   ( q_{m} + p_{_G}^{\perp})|_{(q_{m} + p_{_G}^{\perp})B(H)(q_{m} + p_{_G}^{\perp})} - (q_{m} + p_{_G}^{\perp}) [z_{1} , ](q_{m} + p_{_G}^{\perp}).$$ Then $\Delta_2$ is a weak-2-local derivation on $(q_{m} + p_{_G}^{\perp})B(H)(q_{m} + p_{_G}^{\perp})$ and $\Delta_2 (a) =0$ for every $a\in p_{_G}^{\perp} B(H) p_{_G}^{\perp}.$  Proposition \ref{p new tech one} proves that $\Delta_2$ is a linear $^*$-derivation on $(q_{m} + p_{_G}^{\perp})B(H)(q_{m} + p_{_G}^{\perp})$, which implies the same conclusion for the mapping $(q_{m} + p_{_G}^{\perp}) \Delta   ( q_{m} + p_{_G}^{\perp})|_{(q_{m} + p_{_G}^{\perp})B(H)(q_{m} + p_{_G}^{\perp})}$.\smallskip

If we set $\displaystyle G_1= \left(\sum_{j=1}^{m-1} q_{j}\right) (H) \subsetneq G$, we conclude that the set $$\mathcal{D}\hbox{iam}_{G_1^{\perp}}^-=\left\{ \hbox{diam} (\sigma(w_{_F})) : w_F\in[z_{_F}], \ F\in \mathfrak{F} (H), \ p_{_{F}} \leq p_{_{G_1}}^{\perp} \right\}$$ is bounded (just apply that $p_{_{G_1}}^{\perp} \Delta   p_{_{G_1}}^{\perp} |_{p_{_{G_1}}^{\perp}B(H)p_{_{G_1}}^{\perp}}$ is a bounded linear $^*$-derivation). If we apply the above reasoning to $G_1$, $p_{m-1}$, and $\Delta$, we deduce that $$(q_{m-1} + p_{_{G_1}}^{\perp}) \Delta   ( q_{m-1} + p_{_{G_1}}^{\perp})|_{(q_{m-1} + p_{_{G_1}}^{\perp})B(H)(q_{m-1} + p_{_{G_1}}^{\perp})}$$ is a bounded linear $^*$-derivation. Repeating these arguments a finite number of steps we prove that $\Delta$ is a bounded linear $^*$-derivation.
\end{proof}

The key technical result needed in our arguments follows now as a direct consequence of the preceding proposition.

\begin{corollary}\label{prop unboundedness in a sequences of mutually orthognal finite dimensional subsets} Let $\Delta : B(H) \to B(H)$ be a weak-2-local derivation with $\Delta^{\sharp} = \Delta$. Suppose that the set $\mathcal{D}\hbox{iam} (\Delta)=\left\{ \hbox{diam} (\sigma(w_{_F})) : w_F\in[z_{_F}], \ F\in \mathfrak{F} (H) \right\}$ is unbounded. Then there exists a sequence $(F_n)\subset\mathfrak{F} (H)$ such that $p_{_{F_n}} \perp p_{_{F_m}}$ for every $n\neq m$, and $\displaystyle\hbox{diam} (\sigma(\widetilde{z}_{_{F_n}})) \geq 4^n$ for every natural $n$.
\end{corollary}

\begin{proof} If there exists $G\in \mathfrak{F} (H)$ such that $$\mathcal{D}\hbox{iam}_{G^{\perp}}^-=\left\{ \hbox{diam} (\sigma(w_{_F})) : w_F\in[z_{_F}], \ F\in \mathfrak{F} (H), \ p_{_{F}} \leq p_{_{G}}^{\perp} \right\}$$ is bounded, then Theorem \ref{t boundedness of diamzF in the orthogonal complement of a finite rank projection} implies that $\Delta$ is a linear $^*$-derivation, which contradicts the unboundedness of the set $$\mathcal{D}\hbox{iam} (\Delta)=\left\{ \hbox{diam} (\sigma(w_{_F})) = \left\| p_{_F} \Delta p_{_F}|_{p_{_F}B(H) p_{_F}} \right\| : w_F\in[z_{_F}], \ F\in \mathfrak{F} (H) \right\}.$$ We can therefore assume that $\mathcal{D}\hbox{iam}_{G^{\perp}}^-$ is unbounded for every $G\in \mathfrak{F} (H)$.\smallskip

We shall argue by induction. Let us fix $F_1 \in \mathfrak{F} (H)$ with $\displaystyle\hbox{diam} (\sigma(\widetilde{z}_{_{F_1}})) \geq 4.$ In the notation employed before, the set $\mathcal{D}\hbox{iam}_{{F_1}^{\perp}}^-$ is unbounded. The mapping $ p_{_{F_1}}^{\perp} \Delta p_{_{F_1}}^{\perp}|_{p_{_{F_1}}^{\perp} B(H) p_{_{F_1}}^{\perp}} : p_{_{F_1}}^{\perp} B(H) p_{_{F_1}}^{\perp} \to p_{_{F_1}}^{\perp} B(H) p_{_{F_1}}^{\perp}$ is a weak-2-local derivation and a symmetric mapping (compare \cite[Proposition 2.7]{NiPe2014}). Therefore the set $\mathcal{D}\hbox{iam} (p_{_{F_1}}^{\perp} \Delta p_{_{F_1}}^{\perp}|_{p_{_{F_1}}^{\perp} B(H) p_{_{F_1}}^{\perp}})$ must be unbounded. We can find $F_2\in \mathfrak{F} (H)$ with $p_{_{F_2}}\perp p_{_{F_1}}$ and $\displaystyle\hbox{diam} (\sigma(\widetilde{z}_{_{F_2}})) \geq 4^2.$\smallskip

Suppose we have defined $F_1,\ldots, F_n$ satisfying the desired conditions. Set $K_n:=F_1 \oplus^{\ell_2}\ldots \oplus^{\ell_2} F_n\in \mathfrak{F}(H)$. According to the arguments at the beginning of the proof, $\mathcal{D}\hbox{iam}_{{K_n}^{\perp}}^-$ is unbounded. Therefore, we can find $F_{n+1} \in \mathfrak{F}(H)$ such that $p_{_{F_{n+1}}} \perp p_{_{F_j}}$ for every $j=1,\ldots, n$ and $\displaystyle\hbox{diam} (\sigma(\widetilde{z}_{_{F_{n+1}}})) \geq 4^{n+1}.$
\end{proof}

We shall show next that every weak-2-local derivation on $B(H)$ is bounded on the lattice of projections of $B(H)$.

\begin{theorem}\label{thm boundedness of ztildeF} Let $\Delta : B(H) \to B(H)$ be a weak-2-local derivation with $\Delta^{\sharp} = \Delta$. Then the following statements hold:
\begin{enumerate}[$(a)$] \item The set $\mathcal{D}\hbox{iam} (\Delta)=\left\{ \hbox{diam} (\sigma(w_{_F})) : w_F\in[z_{_F}], \ F\in \mathfrak{F} (H) \right\}$ is bounded;
\item The set $\left\{ \|\widetilde{z}_{_F}\| : \ F\in \mathfrak{F} (H) \right\}$ is bounded;
\end{enumerate}
Consequently, by Alaoglu's theorem, we can find $z_0\in B(H)$ with $z_0=-z_0^*$ and a subnet $(\widetilde{z}_{_F})_{F\in \Lambda}$ of $(\widetilde{z}_{_F})_{F\in \mathfrak{F} (H)}$ converging to $z_0$ in the weak$^*$-topology of $B(H)$.
\end{theorem}

\begin{proof} $(a)$ Arguing by contradiction, we suppose that $\mathcal{D}\hbox{iam} (\Delta) $ is unbounded. By Corollary \ref{prop unboundedness in a sequences of mutually orthognal finite dimensional subsets}, there exists a sequence $(F_n)\subset\mathfrak{F} (H)$ such that $p_{_{F_n}} \perp p_{_{F_m}}$ for every $n\neq m$, and $\displaystyle\hbox{diam} (\sigma(\widetilde{z}_{_{F_n}})) \geq 4^n$ for every natural $n$. We can pick a sequence of mutually orthogonal rank one projections $(p_k)\subseteq B(H)$ satisfying $p_{2n-1},p_{2n}\leq p_{_{F_n}}$, $\widetilde{z}_{_{F_n}} = i \lambda_{2n-1} p_{2n-1} + i \lambda_{2n} p_{2n} + (p_{_{F_n}} -p_{2n-1} -p_{2n}) \widetilde{z}_{_{F_n}} (p_{_{F_n}} -p_{2n-1} -p_{2n})$ ($\lambda_{2n-1},\lambda_{2n}\in \mathbb{R}$), and $|\lambda_{2n-1} - \lambda_{2n}|= \lambda_{2n-1} - \lambda_{2n} = \hbox{diam} (\sigma(\widetilde{z}_{_{F_n}})) \geq 4^n$.\smallskip

Let $e_n$ be the unique rank-2 partial isometry in $B(H)$ defined by $e_{n} = \xi_{2n} \otimes \xi_{2n-1} +\xi_{2n-1}\otimes \xi_{2n}$, where $\xi_{2n}$ and $\xi_{2n-1}$ are norm one vectors in $p_{2n} (H)$ and $p_{2n-1}(H)$, respectively. Since $e_n\perp e_m$, for every $n\neq m$, the series $\displaystyle \sum_{n=1}^{\infty} e_{n}$ converges to an element $a_0\in B(H)$.  Set $\displaystyle s_{2n}:= \sum_{k=1}^{2n} p_k\leq  p_{_{K_n}}$, where $\displaystyle K_n =\bigoplus_{k=1}^{n} F_{k}$.  Clearly, $a_0 = s_{2n} a_0 s_{2n} +  s_{2n}^{\perp} a_0 s_{2n}^{\perp}$. Applying the properties of $\widetilde{z}_{_{F_n}}$ (compare Lemma \ref{l zF}, Remark \ref{remark [zF]} and Lemma \ref{l F1 in F2}) and \cite[Lemma 3.2]{NiPe2015} we have $$s_{2n} \Delta (a_0) s_{2n} = s_{2n} \Delta (s_{2n}a_0 s_{2n}) s_{2n} = s_{2n} [\widetilde{z}_{_{K_n}} ,s_{2n}a_0 s_{2n}] s_{2n} $$ $$= \left[i \sum_{k=1}^{n} \lambda_{2k-1} p_{2k-1} + \lambda_{2k} p_{2k}, s_{2n}a_0 s_{2n}\right].$$ \smallskip

Let us consider consider the functional $\phi_{0}= \sum_{k=1}^{n} \frac{1}{2^{k}} \omega_{\xi_{2k-1},\xi_{2k}}$, where, following the standard notation,  $\omega_{\xi_{2k-1},\xi_{2k}} (a) = \langle \xi_{2k-1}, a(\xi_{2k}) \rangle$ ($a\in B(H)$). We deduce from the above that $\|\phi_0\|\leq 1$ and  $$ \|\Delta(a_0)\| \geq \left| \phi_0 (s_{2n} \Delta (a_0) s_{2n}) \right| = \sum_{k=1}^{n} \frac{1}{2^{k}} (\lambda_{2k-1} - \lambda_{2k}) $$ $$= \sum_{k=1}^{n} \frac{1}{2^{k}} |\lambda_{2k-1} - \lambda_{2k}| > \sum_{k=1}^{n} \frac{1}{2^{k}} 4^{k} = \sum_{k=1}^{n} 2^{k},$$ which is impossible.\smallskip

$(b)$ Take $F\in \mathfrak{F} (H)$ and any $z\in [z_{_F}]$. If we choose $i \lambda\in \sigma (z_{_F})$, the inequalities $$\| \widetilde{z}_{_F}\| \leq \| z - i \lambda p_{_F}\| \leq \hbox{diam}(\sigma(z - i \lambda p_{_F}))= \hbox{diam}(\sigma(z)) = \hbox{diam}(\sigma(\widetilde{z}_{_F} )),$$ hold
because $0\in\sigma(z - i \lambda p_{_F})$ and $(z - i \lambda p_{_F})^*=- (z - i \lambda p_{_F})$. Finally, the desired conclusion follows from statement $(a)$.
\end{proof}

We can provide now a positive answer to Problem \ref{problem 2 new} in the case $A= B(H)$.\smallskip

\begin{theorem}\label{t weak2local derivations on BH which are symmetric are derivations}
Let $\Delta : B(H) \to B(H)$ be a weak-2-local derivation with $\Delta^{\sharp} = \Delta$. Then $\Delta$ is a linear $^*$-derivation.
\end{theorem}

\begin{proof} By Theorem \ref{thm boundedness of ztildeF}, the set $$\mathcal{D}\hbox{iam} (\Delta)=\left\{ \hbox{diam} (\sigma(w_{_F})) : w_F\in[z_{_F}], \ F\in \mathfrak{F} (H) \right\}$$ is bounded. The desired conclusion follows from Corollary \ref{c id principle 3}.
\end{proof}

All the results from Lemma \ref{l zF} to Proposition \ref{prop unboundedness in a sequences of mutually orthognal finite dimensional subsets} remain valid when $\Delta : K(H) \to K(H)$ is a weak-2-local derivation with $\Delta^{\sharp} = \Delta$. Actually, the conclusion of Theorem \ref{thm boundedness of ztildeF} also holds for every such a mapping $\Delta$ with practically the same proof, but replacing $a_0=\displaystyle \sum_{n=1}^{\infty} e_{n}\in B(H)$ with $a_0 = \displaystyle \sum_{n=1}^{\infty} \left(\frac{2}{3}\right)^n e_{n}\in K(H)$, because in that case we would have $$ \|\Delta(a_0)\| \geq \left| \phi_0 (s_{2n} \Delta (a_0) s_{2n}) \right| = \sum_{k=1}^{n} \frac{1}{2^{k}} \left(\frac{2}{3}\right)^k |\lambda_{2k-1} - \lambda_{2k}|> \sum_{k=1}^{n}  \left(\frac{4}{3}\right)^k,$$ obtaining the desired contradiction. We have thus obtained an appropriate version of Theorem \ref{thm boundedness of ztildeF} for weak-2-local derivations on $K(H)$.

\begin{theorem}\label{thm boundedness of ztildeF compact} Let $\Delta : K(H) \to K(H)$ be a weak-2-local derivation with $\Delta^{\sharp} = \Delta$. Then the following statements hold:
\begin{enumerate}[$(a)$] \item The set $\mathcal{D}\hbox{iam} (\Delta)=\left\{ \hbox{diam} (\sigma(w_{_F})) : w_F\in[z_{_F}], \ F\in \mathfrak{F} (H) \right\}$ is bounded;
\item The set $\left\{ \|\widetilde{z}_{_F}\| : \ F\in \mathfrak{F} (H) \right\}$ is bounded;
\end{enumerate}
Consequently, by Alaoglu's theorem, we can find $z_0\in B(H)$ with $z_0=-z_0^*$ and a subnet $(\widetilde{z}_{_F})_{F\in \Lambda}$ of $(\widetilde{z}_{_F})_{F\in \mathfrak{F} (H)}$ converging to $z_0$ in the weak$^*$-topology of $B(H)$. $\hfill\Box$
\end{theorem}

Applying a subtle adaptation of the previous arguments we get the following.

\begin{theorem}\label{t weak2local derivations on KH which are symmetric are derivations}
Let $\Delta : K(H) \to K(H)$ be a weak-2-local derivation with $\Delta^{\sharp} = \Delta$. Then $\Delta$ is a linear $^*$-derivation.
\end{theorem}

\section{weak-2-local derivations on $B(H)$}

We can culminate now the study of weak-2-local derivations on $B(H)$ with the promised solution to Problem \ref{problem 1 new} in the case $A=B(H)$.

\begin{theorem}\label{t final} Let $H$ be an arbitrary complex Hilbert space, and let $\Delta$ be a weak-2-local derivation on $B(H)$. Then $\Delta$ is a linear derivation.
\end{theorem}

\begin{proof} We have already commented that $H$ can be assumed to infinite dimensional. Suppose $\Delta : B(H)\to B(H)$ is a weak-2-local derivation. Since the set $\mathcal{S} = \hbox{Der}(A)$, of all derivations on $B(H)$, is a linear subspace of $B(B(H))$, we deduce from Lemma \ref{l new basic properties 1}$(c)$ and $(e)$ that $\Delta_1 = \frac{\Delta + \Delta^{\sharp}}{2}$ and $\Delta_2 = \frac{\Delta - \Delta^{\sharp}}{2 i}$ are weak-2-local derivations on $B(H)$. Since $\Delta_1 = \Delta_1^{\sharp}$ and $\Delta_2 = \Delta_2^{\sharp}$, Theorem \ref{t id princ} proves that $\Delta_1$ and $\Delta_2$ are linear $^*$-derivations on $B(H)$, and thus, $\Delta = \Delta_1 + i  \Delta_2$ is a linear derivation on $B(H)$.
\end{proof}

According to Theorem \ref{t weak2local derivations on KH which are symmetric are derivations}, the arguments developed to prove Theorem \ref{t final} are also valid to obtain the following:

\begin{theorem}\label{t final KH} Let $H$ be an arbitrary complex Hilbert space, and let $\Delta$ be a weak-2-local derivation on $K(H)$. Then $\Delta$ is a linear derivation.$\hfill\Box$
\end{theorem}

We begin with a suitable generalization of \cite[Lemma 3.2]{NiPe2015}.

\begin{lemma}\label{l orthogonal sums} Let $A_1$ and $A_2$ be C$^*$-algebras, and let $\Delta : A_1 \oplus^{\infty} A_2 \to A_1 \oplus^{\infty} A_2$ be a weak-2-local derivation.
Then $\Delta (A_j)\subseteq A_j$ for every $j=1,2$. Moreover, if $\pi_j $ denotes the projection of $ A_1 \oplus^{\infty} A_2$ onto $A_j$, we have $\pi_j \Delta (a_1+a_2) = \pi_j \Delta (a_j)$, for every $a_1\in A_1$, $a_2\in A_2$ and $j=1,2$.
\end{lemma}

\begin{proof} Let us fix $a_1\in A_1$. Every C$^*$-algebra admits a bounded approximate unit (cf.
\cite[Theorem 1.4.2]{Ped}), thus, by Cohen's factorisation theorem (cf. \cite[Theorem VIII.32.22 and Corollary VIII.32.26]{HewRoss}), there exist $b_1,c_1\in A_1$ satisfying $a_1 = b_1 c_1$. We recall that $A^{*} = A_1^* \oplus^{\ell_1} A_2^*$. By hypothesis, for each $\phi\in A_2^{*}$, there exists a derivation $D_{a_1,\phi}: A_1 \oplus^{\infty} A_2 \to A_1 \oplus^{\infty} A_2$ satisfying $$\phi \Delta_{a_1,\phi} (a_1) = \phi D_{a_1,\phi}(a_1) = \phi  D_{a_1,\phi}(b_1 c_1) = \phi ( D_{a_1,\phi}(b_1 ) c_1) + \phi  (b_1 D_{a_1,\phi}(c_1)) = 0,$$ where in the last equalities we applied that $D_{a_1,\phi}(b_1 ) c_1$ and  $b_1 D_{a_1,\phi}(c_1)$ both lie in $A_1$ and $\phi\in A_2^*$. We deduce, via Hahn-Banach theorem, that $\Delta(a_1) \in A_1$.\smallskip

The above arguments also show that, for each derivation $D: A_1 \oplus^{\infty} A_2 \to A_1 \oplus^{\infty} A_2$ we have $D (A_j)\subseteq A_j$ for every $j=1,2$. It follows from the hypothesis that, for each  $\phi\in A_1^{*}$, $a_1\in A_1$ and $a_2\in A_2$, there exists a derivation $D_{\phi,a_1+a_2, a_1}: A_1 \oplus^{\infty} A_2 \to A_1 \oplus^{\infty} A_2$ satisfying $$\phi \Delta (a_1) = \phi D_{\phi,a_1+a_2, a_1}(a_1), \hbox{ and } \phi \Delta (a_1+a_2) = \phi D_{\phi,a_1+a_2, a_1}(a_1+a_2).$$ In particular, $\phi \Delta (a_1) = \phi \Delta (a_1+a_2)$, for every $\phi\in A_1^*$. It follows that $\pi_1 \Delta (a_1) = \pi_1 \Delta (a_1+a_2)$.
\end{proof}

For further purposes, we shall also explore the stability of the above results under $\ell_{\infty}$- and $c_0$-sums.

\begin{proposition}\label{p ellinftysums} Let $(A_{j})$ be an arbitrary family of C$^*$-algebras. Suppose that for each $j$, every weak-2-local derivation on $A_j$ is a linear derivation. Then the following statements hold:
\begin{enumerate}[$(a)$] \item Every weak-2-local derivation on $A= \bigoplus^{\ell_{\infty}} A_j$ is a linear derivation;
\item Every weak-2-local derivation on $A= \bigoplus^{c_0} A_j$ is a linear derivation.
\end{enumerate}
\end{proposition}

\begin{proof} $(a)$ Let $\Delta : \bigoplus^{\ell_{\infty}}_{j\in J} A_j \to \bigoplus^{\ell_{\infty}}_{j\in J} A_j$ be a weak-2-local derivation. Let $\pi_j$ denote the natural projection of $A$ onto $A_j$. If we fix an index $j_0\in J$, it follows from Lemma \ref{l orthogonal sums} that $\Delta (A_{j_0}) \subseteq A_{j_0}$ and $\displaystyle \Delta \left( \bigoplus^{\ell_{\infty}}_{j_0\neq j\in J} A_j \right) \subseteq \bigoplus^{\ell_{\infty}}_{j_0\neq j\in J} A_j$. We deduce from the assumptions that $\Delta|_{A_{j}} : A_{j} \to A_{j}$ is a linear derivation for every $j$.\smallskip

We shall finish the proof by showing that $\{ \|\Delta|_{A_{j}}\| : j\in J\}$ is a bounded set. Otherwise, there exist infinite sequences $(j_n)\subseteq J,$ $(a_{j_n})\subset A,$ with $a_{j_n}\in A_{j_n},$ $\|a_{j_n}\|\leq 1$, and $\|\Delta (a_{j_n})\| > 4^{n}$, for every natural $n$. Let $\displaystyle a_0 = \sum_{n=1}^{\infty} a_{j_n}\in A$. For each natural $n$, $a_0 = a_{j_n} + (a_0-a_{j_n})$ with $a_{j_n} \perp (a_0-a_{j_n})$ in $A$. It follows from the above properties and the second statement in Lemma \ref{l orthogonal sums} that $$ \| \Delta (a_0) \| \geq \|\pi_{j_n} \Delta (a_0) \|= \| \Delta (a_{j_n}) \| > 4^n,$$ for every $n\in \mathbb{N}$, which is impossible.\smallskip

$(b)$ The proof of $(a)$ but replacing $\displaystyle a_0 = \sum_{n=1}^{\infty} a_{j_n}$ with $\displaystyle a_0 = \sum_{n=1}^{\infty} \frac{1}{2^n} a_{j_n}\in A$ remains valid in this case.
\end{proof}

Following standard notation, we shall say that a von Neumann algebra $M$ is atomic if $M = \bigoplus^{\ell_{\infty}} B(H_{\alpha})$, where each $H_{\alpha}$ is a complex Hilbert space.
We recall that a Banach algebra is called \emph{dual} or \emph{compact} if, for every $a \in A$, the operator $A\to A$, $b \mapsto aba$ is
compact. By \cite{Al}, compact C$^*$-algebras are precisely the
algebras of the form $(\bigoplus_{i \in I} K(H_i))_{c_0}$, where
each $H_i$ is a complex Hilbert space.\smallskip

We finish this note with a couple of corollaries which follow straightforwardly from Theorems \ref{t final}, \ref{t final KH} and Proposition \ref{p ellinftysums}.

\begin{corollary}\label{t final atomic von Neumann} Every weak-2-local derivation on an atomic von Neumann algebra is a linear derivation.$\hfill\Box$
\end{corollary}

\begin{corollary}\label{t final compact Cstar} Every weak-2-local derivation on a compact C$^*$-algebra is a linear derivation.$\hfill\Box$
\end{corollary}

\end{document}